\documentclass[reqno, 10pt]{amsart}
\usepackage{amssymb}
\usepackage{amsmath}
\usepackage[mathscr]{euscript}
\usepackage{hyperref}
\usepackage{graphicx}
\usepackage{mathabx}
\usepackage{comment}
\usepackage{amssymb}
\usepackage{enumitem}

\usepackage{hyperref}
\RequirePackage[dvipsnames]{xcolor}
\definecolor{halfgray}{gray}{0.55} 
\definecolor{webgreen}{rgb}{0,0.5,0}
\definecolor{webbrown}{rgb}{.6,0,0} 
\hypersetup{
  colorlinks=true, linktocpage=true, pdfstartpage=3,
  pdfstartview=FitV,
  breaklinks=true, pdfpagemode=UseNone, pageanchor=true,
  pdfpagemode=UseOutlines,
  plainpages=false, bookmarksnumbered, bookmarksopen=true,
  bookmarksopenlevel=1,
  hypertexnames=true,
  pdfhighlight=/O,
  urlcolor=webbrown, linkcolor=RoyalBlue,
  citecolor=webgreen,   pdftitle={},
  pdfauthor={},
  pdfsubject={2000 Mathematical Subject Classification: Primary:},%
  pdfkeywords={},
  pdfcreator={pdfLaTeX},
  pdfproducer={LaTeX with hyperref}
}

\theoremstyle{thmstyleone}
\newtheorem{theorem}{Theorem}[section]
\newtheorem{proposition}[theorem]{Proposition}

\theoremstyle{thmstyletwo}
\newtheorem{example}[theorem]{Example}

\theoremstyle{thmstylethree}
\newtheorem{definition}[theorem]{Definition}

\begin{document}

\title[Spectral characterization of shadowing]{Spectral characterization of shadowing for linear operators on Hilbert spaces}

\author[M. Pituk]{Mih\'{a}ly Pituk}
\address{Department of Mathematics, University of Pannonia, Egyetem \'ut 10, 8200 Veszpr{\'e}m, Hungary;
HUN--REN--ELTE Numerical Analysis and Large Networks Research Group, Budapest, Hungary}
\email{pituk.mihaly@mik.uni-pannon.hu}



\begin{abstract}
\vskip20pt
In this paper, we study one of the fundamental notions in dynamical systems, the shadowing of invertible (bounded and linear) operators on a Hilbert space. Although the problem of finding a spectral characterization for shadowing has been in the focus of the research for a long time, spectral criteria are available 
only for rather special classes of invertible operators. In this paper, we give a complete spectral characterization for the shadowing of an arbitrary invertible operator~\emph{T} on a complex Hilbert space. It is shown that~\emph{T} has the shadowing property if and only if its right spectrum is disjoint from the unit circle in the complex plane. As a consequence, the shadowing property for~\emph{T} is equivalent to the uniform expansivity of its adjoint operator.
\end{abstract}

\date{March 21, 2026}

\maketitle

\section{Introduction and the main result}\label{intro}
In this paper, we are interested in the shadowing of invertible (bounded and linear) operators on Hilbert spaces. \emph{Shadowing\,}, together with \emph{hyperbolicity} and \emph{uniform expansivity}, is a fundamental notion in the theory of dynamical systems (see, e.g., \cite{Diam}, \cite{Pal}, \cite{Pil}). 
It is well-known that an invertible operator on a complex Banach space is hyperbolic (uniformly expansive) if and only if its spectrum (approximate point spectrum) is disjoint from the unit circle in the complex plane. The aim of this paper is to give a similar spectral characterization for shadowing. In Theorem~\ref{thm:specshad}, 
we establish the following result: \emph{An invertible operator on a complex Hilbert space has the shadowing property if and if its right spectrum is disjoint from the unit circle in the complex plane.\,} First we introduce some notations and definitions.

Throughout the paper, 
$\mathbb Z$ and $\mathbb C$ denote the set of integers and the set of complex numbers, respectively. Let $\mathbb T=\{\,\lambda\in\mathbb C\mid|\lambda|=1\,\}$, the unit circle in~$\mathbb C$.
Given a complex Banach space~($X,\|\cdot\|)$, the symbol $L(X)$ denotes the space of bounded linear operators $T\colon X\rightarrow X$ equipped with the operator norm
\begin{equation*}
	\|T\|=\sup_{x\in S(X)}\|Tx\|,\qquad T\in L(X),
\end{equation*}
where $S(X)=\{\,x\in X: \|x\|=1\,\}$, the unit sphere in~$X$.

An operator $T\in L(X)$ is called \emph{invertible} if there exists $S\in L(X)$ such that $ST=TS=I$, where $I$ denotes the identity operator on~$X$. The set of all invertible operators in~$L(X)$ will be denoted by~$GL(X)$.

An operator $T\in L(X)$ is called \emph{right invertible} (\emph{left invertible}) if there exists $S\in L(X)$ such that $TS=I$ ($ST=I$).

The \emph{spectrum}, \emph{right spectrum}, \emph{left spectrum}, \emph{point spectrum}, \emph{approximate point spectrum} and the \emph{spectral radius} of $T\in L(X)$ are defined by
\begin{align*}
\sigma(T)&=\{\,\lambda\in\mathbb C\mid\text{$\lambda I-T$ is not invertible}\,\},\\
\sigma_r(T)&=\{\,\lambda\in\mathbb C\mid\text{$\lambda I-T$ is not right invertible}\,\},\\
\sigma_l(T)&=\{\,\lambda\in\mathbb C\mid\text{$\lambda I-T$ is not left invertible}\,\},\\
\sigma_p(T)&=\{\,\lambda\in\sigma(T)\mid\text{$\lambda I-T$ is not injective}\,\},\\
\sigma_a(T)&=\{\,\lambda\in\sigma(T)\mid\text{$\lambda I-T$ is not bounded below}\,\},\\
\end{align*}
and
$r(T)=\sup_{\lambda\in\sigma(T)}|\lambda|$, respectively.
%\begin{equation*}
%r(T)=\sup_{\lambda\in\sigma(T)}|\lambda|,
%\end{equation*}
%respectively. 
An operator $T\in L(X)$ is \emph{bounded below} if 
there exists $\alpha>0$ such that $\|Tx\|\geq\alpha\|x\|$ for all $x\in X$.

In the following definitions of shadowing, hyperbolicity and uniform expansivity, $X$ is a complex Banach space.

\vskip10pt
\begin{definition}\label{def:shad}
{\rm 
Let $\delta>0$. By a \emph{$\delta$-pseudo\-trajectory} of an operator~$T\in GL(X)$, we mean a sequence $(x_n)_{n\in\mathbb Z}$ in~$X$ such that 
\begin{equation*}
\|x_{n+1}-Tx_n\|\leq\delta\qquad\text{for all $n\in\mathbb Z$.} 
\end{equation*}
We say that $T$ has the \emph{shadowing property} if for every $\epsilon>0$ there exists $\delta>0$ such that every $\delta$-pseudotrajectory of $(x_n)_{n\in\mathbb Z}$ of~$T$ is $\epsilon$-shadowed by a real trajectory of~$T$, i.e., there exists $x\in X$ such that
\begin{equation*}
\|x_n-T^n x\|\leq\epsilon\qquad\text{for all $n\in\mathbb Z$.} 
\end{equation*}
}
\end{definition} 

\vskip10pt
\begin{definition} \label{def:hyp}
{\rm
An operator $T\in GL(X)$ 
is called \emph{hyperbolic} if $X$ can be decomposed into the direct sum $X=M\oplus N$, where $M$ and~$N$ are closed  
subspaces of~$X$ with 
$T(M)=M$, $T(N)=N$ and such that the restrictions 
$T|_M\in L(M)$ and $T^{-1}|_N\in L(N)$ are uniform contractions, i.e.,
$r(T|_M)<1$ and $r(T^{-1}|_N)<1$.
}
\end{definition} 

\vskip10pt
\begin{definition} \label{def:unexp}\hskip-2.1pt
{\rm
An operator $T\in GL(X)$  is called \emph{uniformly expansive} if there exists a positive integer~$n$ such that, for every $x\in S(X)$, $\|T^n(x)\|\geq2$ or $\|T^{-n}(x)\|\geq2$.
}
\end{definition} 
\vskip10pt

In finite-dimensional spaces and for normal operators on Hilbert spaces the above three notions are equivalent (see~\cite{Ber}, \cite{Eis}, \cite{Omb}). Hyperbolicity always implies 
uniform expansivity (see \cite[Corollary~1]{Eis}) and shadowing (see \cite[Theorem~1]{Omb}), but, in general, the converse implications are false. Examples of nonhyperbolic uniformly expansive operators were given by Eisenberg and Hedlund~\cite
{Eis}, while Bernardes~{\it et al.}~\cite
{Ber} have constructed nonhyperbolic operators with the shadowing property. For further examples of nonhyperbolic operators with the shadowing property, see ~\cite{Ani}, \cite{Cir} and Example~\ref{ex:illust}.
We should also mention an important recent result, established independently by Bernardes and Messaoudi~\cite[Theorem~1]{BeMe} and by Cirillo~\emph{et al.}~\cite[Theorem~9]{Cir}, which shows that an invertible operator on a complex Banach space is hyperbolic if and only if it has the shadowing property and is (uniformly) expansive.
For further results and a list of related open problems, see
\cite{Ani}, \cite{BerCar} and~\cite{BePe}.

As noted before, hyperbolicity and uniform expansivity can be characterized in terms of the spectral values of operator~$T$ as follows.
\vskip10pt
\begin{theorem}\label{thm:spechyp}
{\rm \cite[Lemma~1]{Eis}}
Let $X$ be a complex Banach space. An operator $T\in GL(X)$ is hyperbolic if and only if
\begin{equation}\label{eq:spechyp}
\sigma(T)\cap\mathbb T=\emptyset,
\end{equation}
where $\sigma(T)$ is the spectrum of~$T$.	
\end{theorem}
\vskip10pt

\begin{theorem}\label{thm:specuniexp}
{\rm \cite[Theorem~1]{Hed}}
Let $X$ be a complex Banach space. An operator $T\in GL(X)$ is uniformly expansive if and only if
\begin{equation}\label{eq:specuniexp}
\sigma_a(T)\cap\mathbb T=\emptyset,
\end{equation}
where $\sigma_a(T)$ is the approximate point spectrum of~$T$.
\end{theorem}
\vskip10pt
It is well-known (see
%~\cite{Con} and/or 
Proposition~\ref{prop:rightsprop}) that if~$X$ is a complex Hilbert space and $T\in L(X)$, then $\sigma_a(T)$ coincides with the left spectrum of~$T$. Consequently, in Hilbert spaces, Theorem~\ref{thm:specuniexp} can be reformulated as follows.
\vskip10pt
\begin{theorem}\label{thm:refspecuniexp}
Let $X$ be a complex Hilbert space. An operator $T\in GL(X)$ is uniformly expansive if and only if
\begin{equation}\label{eq:refspecuniexp}
\sigma_l(T)\cap\mathbb T=\emptyset,
\end{equation}
where $\sigma_l(T)$ is the left spectrum of~$T$.
\end{theorem}
\vskip10pt

To the best of our knowledge, a similar spectral characterization for shadowing is not available in the literature. 

Our aim in this paper is to prove the following theorem which provides a complete spectral characterization of the shadowing property for an invertible operator on a complex Hilbert space. 

\vskip10pt
\begin{theorem}\label{thm:specshad}
Let $X$ be a complex Hilbert space. An operator $T\in GL(X)$ has the shadowing property if and only if
\begin{equation}\label{eq:specshad}
\sigma_r(T)\cap\mathbb T=\emptyset,
\end{equation}
where $\sigma_r(T)$ is the right spectrum of~$T$.
\end{theorem}
\vskip10pt

The proof of Theorem~\ref{thm:specshad} will be given in Sec.~\ref{sec:proof}.

Since $\sigma(T)=\sigma_l(T)\cup\sigma_r(T)$, Theorems~\ref{thm:refspecuniexp} and~\ref{thm:specshad} provide, in the Hilbert space setting, an alternative proof of the characterization of hyperbolicity in terms of the shadowing property and uniform expansivity.
\vskip0pt
Let~$X$ be a complex Hilbert space and $T\in L(X)$. Recall that the \emph{adjoint operator of~$T$} is the unique operator ~$T^*\in L(X)$ satisfying 
\begin{equation*}
	\langle Tx,y\rangle=\langle x,T^* y\rangle\qquad\text{for all $x$, $y\in X$},
\end{equation*}
where $\langle\cdot,\cdot\rangle$ denotes the inner product in~$X$. It is known (see~\cite{Con} and Proposition~\ref{prop:rightsprop}) that 
\begin{equation*}
\sigma(T^*)=(\sigma(T))^*
\qquad
\text{and}
\qquad
\sigma_a(T^*)=(\sigma_r(T))^*,	
\end{equation*}
where, for any subset~$\Lambda\subset\mathbb C$, the symbol $\Lambda^*:=\{\,\overline{\lambda}:\lambda\in\Lambda\,\}$ denotes the set of complex conjugates of the elements of~$\Lambda$.
From this and
Theorem~\ref{thm:specuniexp}, we obtain the following equivalent form of Theorem~\ref{thm:specshad}.

\vskip10pt
\begin{theorem}\label{thm:refspecshad}
Let $X$ be a complex Hilbert space. An operator $T\in GL(X)$ has the shadowing property if and only if its adjoint operator~$T^*$ is uniformly expansive, i.e.,
\begin{equation}\label{eq:refspecshad}
\sigma_a(T^*)\cap\mathbb T=\emptyset.
\end{equation}
\end{theorem}
\vskip10pt

For the illustration of Theorems~\ref{thm:specshad} and~\ref{thm:refspecshad}, let us revisit the following example due to Eisenberg and Hedlund~\cite{Eis}.

\vskip10pt
\begin{example}\label{ex:illust}
{\rm 
Let~$X$ be an infinite dimensional separable complex Hilbert space with orthonormal basis $(e_n)_{n\in\mathbb Z}$ and consider the bilateral weighted shifts~$T$ and $S$ of~$X$ defined by
\begin{equation*}
Te_n=
\begin{cases}2\sqrt{2}e_{n+1}\qquad&\text{for $n\geq0$},\\[6pt]
\dfrac{1}{2\sqrt{2}}e_{n+1}\qquad&\text{for $n\leq-1$},\\[6pt]	
\end{cases}
\end{equation*}
and
\begin{equation*}
Se_n=
\begin{cases}2\sqrt{2}e_{n-1}\qquad&\text{for $n\geq1$},\\[6pt]
\dfrac{1}{2\sqrt{2}}e_{n-1}\qquad&\text{for $n\leq0$},\\	
\end{cases}
\end{equation*}
respectively. We will show that~$T$ and~$S$ are nonhyperbolic invertible operators operators, $T$ is uniformly expansive without the shadowing property, while~$S$ has the shadowing property without being uniformly expansive.

It is easy to verify that $S=T^*$ so that $S^*=T$. Eisenberg and Hedlund~\cite[Example~4]{Eis} have shown that
\begin{gather*}
	\sigma(S^*)=\sigma (T)=\biggl\{\,\lambda\in\mathbb C: \frac{1}{2\sqrt{2}}\leq|\lambda|\leq 2\sqrt{2}\,\biggr\},\\[6pt]
	\sigma_p(S^*)=\sigma_p(T)=\emptyset,\\[6pt]
	\sigma_p(S)=\sigma_p(T^*)=\biggl\{\,\lambda\in\mathbb C: \frac{1}{2\sqrt{2}}<|\lambda|<2\sqrt{2}\,\biggr\},
	\end{gather*}
 and
 \begin{equation*}
 \sigma_a(S^*)=\sigma_a(T)=\biggl\{\,\lambda\in\mathbb C: \text{$|\lambda|=\frac{1}{2\sqrt{2}}$\, or\, $|\lambda|=2\sqrt{2}$}\,\biggr\}.
\end{equation*}
Clearly, $\mathbb T\subset\sigma(T)=\sigma(S^*)=(\sigma(S))^*$ and hence $\mathbb T=\mathbb T^*\subset\sigma(S)$. Therefore, neither~$T$ nor~$S$ is hyperbolic.
 Since $\sigma_a(S^*)\cap\mathbb T=\sigma_a(T)\cap\mathbb T=\emptyset$, Theorems~\ref{thm:specuniexp} and~\ref{thm:refspecshad} imply that~$S$ has the shadowing property and $T$ is uniformly expansive.  Finally, from the relations $\mathbb T\subset\sigma_p(S)\subset\sigma_a(S)$ and $\mathbb T\subset\sigma_p(T^*)\subset\sigma_a(T^*)$, using Theorems~\ref{thm:specuniexp} and~\ref{thm:refspecshad} again, we conclude~$S$ is not uniformly expansive and~$T$ does not have the shadowing property. 

For a complete characterization of shadowing and expansivity of the classical bilateral shifts and their spectral properties, see the recent papers by Bernardes~{\it al.}~\cite{Ber}, \cite{BeMe} and D'Aniello~{\it et al.}~\cite{Ani}, \cite{AnMa}.
}
\end{example}
\vskip10pt
The paper is organized as follows. In Sec.~\ref{sec:prelim}, we present several preliminary results 
which will be needed in the proof of Theorem~\ref{thm:specshad}. 
The proof of Theorem~\ref{thm:specshad} is given in Sec.~\ref{sec:proof}.

\section{Preliminaries}\label{sec:prelim}
The following proposition provides characterizations of the one-sided spectra of operators on  Hilbert spaces.

\vskip10pt
\begin{proposition}
\label{prop:rightsprop}
{\rm \cite[Chap.~XI, Par.~1, Proposition~1.1, p.~347]{Con}}\,
If $X$ is a complex Hilbert space, $T\in L(X)$ and $\lambda\in\mathbb C$, then the following statements are equivalent. \begin{enumerate}[label={(\roman*)},itemindent=1em]
\item[\rm (i)] $\lambda\notin\sigma_r(T)$.\label{lm:rightsprop:it1}
\item[\rm (ii)] $\overline{\lambda}\notin\sigma_a(T^*)$.\label{lm:rightsprop:it2}
\item[\rm (iii)] $\lambda I-T$ is surjective.\label{lm:rightsprop:it13}
\item[\rm (iv)] $\overline{\lambda}\notin\sigma_l(T^*)$.\label{lm:rightsprop:it4}
\end{enumerate}
In particular, we have that
\begin{equation*}
	\sigma_r(T)=(\sigma_a(T^*))^*\qquad\text{and}\qquad
	\sigma_l(T)=\sigma_a(T).
	\end{equation*}
\end{proposition}
\vskip10pt

The following simple result will be needed
in the proof of Theorem~\ref{thm:specshad}. Its first part was already observed by Bernardes and Peris~\cite{BePe} (see \cite[Proposition~29~(c) and Remark~45]{BePe}). We present its proof for completeness.

\vskip10pt
\begin{proposition}\label{prop:unitshad}
Let $X$ be a complex Banach space, $T\in GL(X)$ and $\lambda\in\mathbb T$. If $T$ has the shadowing property, then so does $\lambda^{-1}T$. If, in addition, we assume that $X$ is a Hilbert space, then $\lambda\in\sigma_r(T)$ if and only if $1\in\sigma_r(\lambda^{-1}T)$.
\end{proposition}
\vskip10pt

\begin{proof}Let $\lambda\in\mathbb T$ and suppose that $T\in GL(X)$ has the shadowing property.
Let $\epsilon>0$ be arbitrary. 
Choose $\delta>0$ such that every $\delta$-pseudotrajectory of~$T$ is $\epsilon$-shadowed by a real trajectory of~$T$. (Since~$T$ has the shadowing property, such a $\delta$ certainly exists.)
Let $(y_n)_{n\in\mathbb Z}$ be an arbitrary $\delta$-pseudotrajectory of~$\lambda^{-1}T$.
Define $x_n=\lambda^{n}y_n$ for $n\in\mathbb Z$. Since $|\lambda|=1$, we have for $n\in\mathbb Z$,
\begin{equation*}
	\|x_{n+1}-Tx_n\|=\|\lambda^{n+1}(y_{n+1}-\lambda^{-1}T y_n)\|=\|y_{n+1}-\lambda^{-1}T y_n\|\leq\delta. 
\end{equation*}
Therefore, $(x_n)_{n\in\mathbb Z}$ is a $\delta$-pseudotrajectory of~$T$. By the choice of~$\delta$, there exists $x\in X$ such that	
$\|x_n-T^n x\|\leq\epsilon$ for all $n\in\mathbb Z$. Hence,
\begin{equation*}
	\|y_{n}-(\lambda^{-1}T)^n x\|=\|\lambda^{-n}(x_n-T^n x)\|=\|x_n-T^n x\|\leq\epsilon
	\end{equation*}
for all $n\in\mathbb Z$, which shows that $(y_n)_{n\in\mathbb Z}$ is $\epsilon$-shadowed by a real trajectory of~$\lambda^{-1}T$.
Consequently, $\lambda^{-1}T$ has the shadowing property.

Now assume that $X$ is a Hilbert space and $\lambda\in\mathbb T$. Since 
\begin{equation*}
\lambda I-T=\lambda(I-\lambda^{-1}T),
\end{equation*} 
the surjectivity of $\lambda I-T$ is equivalent to that of $I-\lambda^{-1}T$. From this, by the application of Proposition~\ref{prop:rightsprop}, we conclude that $\lambda\notin\sigma_r(T)$ if and only if $1\notin\sigma_r(\lambda^{-1}T)$.
\end{proof}
\vskip10pt

Given a complex Banach space~$X$, let $l_1(X)$ and $l_\infty(X)$ denote the Banach space of all sequences $(x_n)_{n\in\mathbb Z}$ in~$X$ such that
\begin{equation*}
	\|x\|_1:=\sum_{n=-\infty}^{\infty}\|x_n\|<\infty
\end{equation*}
and
\begin{equation*}
	\|x\|_\infty:=\sup_{n\in\mathbb Z}\|x_n\|<\infty,
\end{equation*}
respectively. 
Now suppose that $X$ is a complex Hilbert space and $T\in GL(X)$. Define operators $\mathcal B\colon l_1(X)\rightarrow l_1(X)$ and $\mathcal S\colon l_\infty(X)\rightarrow l_\infty(X)$ by
\begin{equation}\label{eq:bopdef}
\mathcal B \left((x_n)_{n\in\mathbb Z}\right)=(x_{n-1}-T^* x_n\,)_{n\in\mathbb Z}\qquad\text{for $(x_n)_{n\in\mathbb Z}\in l_1(X)$}
\end{equation}
and
\begin{equation}\label{eq:sopdef}
\mathcal S\left((x_n)_{n\in\mathbb Z}\right)=(x_{n+1}-T x_n\,)_{n\in\mathbb Z}\qquad\text{for $(x_n)_{n\in\mathbb Z}\in l_\infty(X)$},
\end{equation}
respectively. It is easily seen that both operators are bounded and linear. The following shadowing criterion will play an important role in the proof of Theorem~\ref{thm:specshad}.

\vskip10pt
\begin{proposition}\label{prop:shadcrit1}
Let $X$ be a complex Hilbert space and $T\in GL(X)$. If operators~$\mathcal B$ and~$\mathcal S$ are defined as above, then the following statements are equivalent.
\begin{enumerate}[label={(\roman*)},itemindent=1em]
\item[\rm (i)] $T$ has the shadowing property.\label{prop:shadcrit1:c1}
\item[\rm (ii)] $\mathcal S$ is surjective.\label{prop:shadcrit1:c2}
\item[\rm (iii)] $\mathcal B$ is bounded below.\label{prop:shadcrit1:c3}
\end{enumerate}
\end{proposition}
\vskip10pt

The first part of the proposition, the equivalence of~(i) and~(ii), 
is well-known (see \cite{Pil}, \cite[Lemma~10]{Ber} and \cite[Proposition~3]{Cir})). The above formulation is taken from \cite[Proposition~3]{Cir}.
The proof of the second part will be based on a surjectivity criterion for Banach space operators formulated in the next proposition. Recall that if $X$ is a complex Banach space and $T\in L(X)$, then the (normed-space) adjoint operator~$T'\in L(X')$ is defined by $T'(f)=f\circ T$ for $f\in X'$, where~$X'$ denotes the dual space of~$X$, the space of all bounded linear forms $f\colon X\rightarrow\mathbb C$.

\vskip10pt
\begin{proposition}\label{prop:surj}
{\rm \cite[Chap.~II, Par.~9, Theorem~5, p.~86]{Mul}}
If $X$ is complex Banach space and $T\in L(X)$, then the following statements are equivalent.
\begin{enumerate}[label={(\roman*)},itemindent=1em]
\item[\rm (i)] $T'$ is surjective.\label{prop:surj:c1}
\item[\rm (ii)] $T$ is bounded below.\label{prop:surj:c1}
\end{enumerate}
\end{proposition}
\vskip10pt

\begin{proof}[Proof of Proposition~\ref{prop:shadcrit1}]
As noted above, we need only to prove the equivalence of~(ii) and~(iii). It is known (see, e.g., K\"othe~\cite[Chap.~5, Par.~26, Sec.~8, p.~359]{Kot}) that the dual space of $l_1(X)$ is $l_\infty(X)$ if we identify each element $(y_n)_{n\in\mathbb Z}\in l_\infty(X)$ with the bounded linear form $y\colon l_1(X)\rightarrow\mathbb C$ 
%on $l_1(X)$ 
defined  by
\begin{equation*}
	y(x)=\sum_{n=-\infty}^\infty \langle x_n,y_n\rangle\qquad\text{for $x=(x_n)_{n\in\mathbb Z}\in l_1(X)$}.
\end{equation*}
By the definition of the adjoint operator, we have for all $y=(y_n)_{n\in\mathbb Z}\in l_\infty(X)$ and $x=(x_n)_{n\in\mathbb Z}\in l_1(X)$,
\begin{align*}
	&(\mathcal B'(y))(x)=(y\circ\mathcal B)(x)=y(\mathcal B(x))
	=\sum_{n=-\infty}^\infty \langle x_{n-1}-T^* x_n,y_n\rangle\\
	&=\sum_{n=-\infty}^\infty\langle x_{n-1},y_n\rangle-\sum_{n=-\infty}^\infty \langle T^*x_n,y_n\rangle
	=\sum_{n=-\infty}^\infty \langle x_n,y_{n+1}\rangle-\sum_{n=-\infty}^\infty \langle x_n,(T^*)^*y_n\rangle\\
	&=\sum_{n=-\infty}^\infty \langle x_n,y_{n+1}-Ty_n\rangle=(\mathcal S(y))(x).
\end{align*}
Since $x\in l_1(X)$ and $y\in l_\infty(X)$ were arbitrary, this shows that $\mathcal B'=\mathcal S$ and the equivalence of~(ii) and~(iii)
follows from Proposition~\ref{prop:surj} with $T=\mathcal B$ and~$X$ replaced with $l_1(X)$.	
\end{proof}
\vskip10pt

The next proposition gives a sufficient condition for shadowing which is interesting in its own right. In the proof of Theorem~\ref{thm:specshad}, we will see that in Hilbert spaces this condition is not only sufficient, but also necessary for shadowing.

\vskip10pt
\begin{proposition}\label{prop:shadcrit2}
Let $X$ be a complex Banach space and $T\in GL(X)$. If there exists 
$B\in L(X)$
such that
\begin{equation}\label{eq:contr1}
r_+=\limsup_{n\to\infty}\root{n}\of{\|T^n B\|}<1	
\end{equation}
and
\begin{equation}\label{eq:contr2}
r_-=\limsup_{n\to\infty}\root{n}\of{\|T^{-n}(I- B)\|}<1,
\end{equation}
then $T$ has the shadowing property.
\end{proposition}
\vskip10pt

The result of Proposition~\ref{prop:shadcrit2} is not really new. Its proof is an adaptation of the proof of a shadowing result established by Bernardes \emph{et al.}~\cite[Theorem~A]{Ber}.
 It can also be deduced from a more general result by Backes, Dragi\v cevi\'c and~Singh~\cite[Theorem~3]{BaDrSi}. For completeness, we give an alternative short proof.
\vskip10pt

\begin{proof}
Let $z=(z_n)_{n\in\mathbb Z}\in l_\infty(X)$.  Define a sequence $x=(x_n)_{n\in\mathbb Z}$ by
\begin{equation}\label{eq:xdef}
x_n=\sum_{k=0}^\infty T^k Bz_{n-k-1}-\sum_{k=1}^\infty T^{-k}(I-B)z_{n+k-1}\qquad\text{for $n\in\mathbb Z$}.
\end{equation}
Choose a positive constant~$q$ such that
\begin{equation*}
	\max\{\,r_+,r_-\,\}<q<1.
\end{equation*}
In view of~\eqref{eq:contr1} and~\eqref{eq:contr2} such a~$q$ certainly exists and for some $K>0$,
\begin{equation*}
\|T^k B\|\leq Kq^k	\qquad\text{for all $k=0,1,2,\dots$}
\end{equation*}
and
\begin{equation*}
\|T^{-k}(I-B)\|\leq Kq^k\qquad\text{for all $k=0,1,2,\dots$}
\end{equation*}
From~\eqref{eq:xdef} and the last two inequalities, we obtain for $n\in\mathbb Z$,
\begin{align*}
\|x_n\|&\leq\sum_{k=0}^\infty\|T^k B\|\|z_{n-k-1}\|+\sum_{k=1}^\infty\|T^{-k}(I-B)\|\|z_{n+k-1}\|\\
&\leq K\|z\|_\infty\biggl(\sum_{k=0}^\infty q^k+\sum_{k=1}^\infty q^k\biggr)=K\frac{1+q}{1-q}\|z\|_\infty.
\end{align*}	
This shows that $x$ is well-defined and $x\in l_\infty(X)$. From~\eqref{eq:xdef}, it follows easily that 
\begin{equation*}
x_{n+1}=T x_n+z_n\qquad\text{for all $n\in\mathbb Z$}.
\end{equation*}
Therefore, $x\in l_\infty(X)$ and $\mathcal S(x)=z$ with~$\mathcal S$ as in~\eqref{eq:sopdef}. Since $z\in l_\infty(X)$ was arbitrary, this proves that~$\mathcal S$ is surjective. According to \cite[Proposition~3]{Cir}, 
the surjectivity of~$\mathcal S$ implies that~$T$ has the shadowing property.
\end{proof}
\vskip10pt

In the proof of Theorem~\ref{thm:specshad}, we shall use some facts about operator-valued holomorphic functions (see, e.g., \cite[Chap.~III, Sec.~3.11]{Hil}). If $X$ is a complex Banach space and
$R\colon A\rightarrow L(X)$ is an operator-valued holomorphic function defined on the annulus 
\begin{equation}\label{eq:annul}
A=\{\,\textcolor{blue}{\lambda}\in\mathbb C\mid r_1<|\lambda-\lambda_0|<r_2\,\}
\end{equation}
with $\lambda_0\in\mathbb C$ and $0\leq r_1<r_2\leq\infty$,
then $R$ can be expanded into a Laurent series
\begin{equation*}\label{eq:genlaur}
R(\lambda)=\sum_{n=-\infty}^\infty C_n(\lambda-\lambda_0)^n,
\quad r_1<|\lambda-\lambda_0|<r_2,
\end{equation*}
where the unique coefficients $C_n\in L(X)$, $n\in\mathbb Z$, can be expressed in the form of a contour integral
\begin{equation}\label{eq:intrep}
C_n=\frac{1}{2\pi i}\int_{\gamma_r}(\lambda-\lambda_0)^{-n-1}R(\lambda)\,d\lambda,\qquad n\in\mathbb Z,
\end{equation}
$\gamma_r$ being any positively oriented circle $|\lambda-\lambda_0|=r$ with radius $r\in(r_1,r_2)$.
This implies that 
\begin{equation}\label{eq:limsup1}
\limsup_{n\to\infty}\root{n}\of{\|C_n\|}\leq\frac{1}{r_2}
\end{equation}
and
\begin{equation}\label{eq:limsup2}
\limsup_{n\to\infty}\root{n}\of{\|C_{-n}\|}\leq\ r_1.
\end{equation}
Indeed, from~\eqref{eq:intrep}, we have for $n\in\mathbb Z$
 and $r\in(r_1,r_2)$, 
 \begin{equation*}
\|C_n\|\leq M(r)\,r^{-n},\qquad\text{where $M(r)=\max_{|\lambda-\lambda_0|=r}\|R(\lambda)\|$}.
\end{equation*}
From this, we find for $n\geq2$ and $r\in(r_1,r_2)$,
\begin{equation*}
\root{n}\of{\|C_n\|}\leq 
 \frac{\root{n}\of {M(r)}}{r}
\end{equation*}
and
\begin{equation*}
\root{n}\of{\|C_{-n}\|}\leq r\root{n}\of{M(r)}.
\end{equation*}
Hence,
\begin{equation*}
\limsup_{n\to\infty}\root{n}\of{\|C_n\|}\leq\frac{1}{r}
\end{equation*}
and
\begin{equation*}
\limsup_{n\to\infty}\root{n}\of{\|C_{-n}\|}\leq\ r
\end{equation*}
whenever  $r\in(r_1,r_2)$. The required inequalities~\eqref{eq:limsup1} and~\eqref{eq:limsup2} follow by passing to the limits as $r\to r_2-$ and $r\to r_1+$, respectively.

Finally, we formulate a useful result about the existence of holomorphic right inverses of holomorphic operator-valued functions due to Allan~\cite{All} in a form given by Ivanov~\cite{Iva}.

\vskip10pt
\begin{proposition}\label{prop:holright}
{\rm \cite[Corollary~3.10]{Iva}}
Suppose that $X$ is a complex Banach space and $G$ a domain (an open and connected set) in the complex plane. If $H\colon G\rightarrow L(X)$
is a holomorphic function such that $H(\lambda)$ is right invertible for every $\lambda\in G$, then there exists a holomorphic function $R\colon G\rightarrow L(X)$ satisfying
$H(\lambda)R(\lambda)=I$ for all $\lambda\in G$. 	
\end{proposition}
\vskip10pt

\section{Proof of the main result}\label{sec:proof}

Now we are in a position to give a proof of Theorem~\ref{thm:specshad}.
\begin{proof}[Proof of Theorem~\ref{thm:specshad}]
\emph{Necessity.\,} Suppose, by the way of contradiction, that~$T$ has the shadowing property and there exists $\lambda\in\sigma_r(T)\cap\mathbb T$. Without loss of generality, we may assume that $\lambda=1$. Otherwise, we pass from $T$ to $\lambda^{-1}T$ which, according to Proposition~\ref{prop:unitshad}, also has the shadowing property and $1\in\sigma_r(\lambda^{-1}T)$. Therefore, from now on, we will assume that $\lambda=1$. Since $1\in\sigma_r(T)$, by the application of Proposition~\ref{prop:rightsprop}, we conclude that $1\in\sigma_a(T^*)$, i.e.,
\begin{equation}\label{eq:notbb}
I-T^*\quad \text{is not bounded below}.
\end{equation}
 By Proposition~\ref{prop:shadcrit1}, the shadowing property of~$T$ implies that operator~$\mathcal B$ given by~\eqref{eq:bopdef} is bounded below. Therefore, there exists $\alpha>0$ such that
 \begin{equation}\label{eq:bbineq}
 \|\mathcal B(y)\|_1\geq\alpha\|y\|_1\qquad\text{for all $y=(y_n)_{n\in\mathbb Z}\in l_1(X)$}.
\end{equation}
Let $x\in X$ and $q\in(1,\infty)$. Define a sequence $y=(y_n)_{n\in\mathbb Z}$ in~$X$ by
\begin{equation}\label{eq:ytest}
y_n=
\begin{cases} q^n x\qquad&\text{for $n<0$}\\
 q^{-n}x\qquad&\text{for $n\geq0$}.\\
\end{cases}	
\end{equation}
Since $q>1$, we have that
\begin{equation*}
\|y\|_1=\sum_{n=-\infty}^\infty\|y_n\|=\|x\|
\biggl(\,\sum_{n=1}^\infty q^{-n}+\sum_{n=0}^\infty q^{-n}\biggr)=\frac{1+q}{q-1}\|x\|<\infty.
\end{equation*}
From~\eqref{eq:bopdef} and~\eqref{eq:ytest}, we find that
\begin{align*}
\|\mathcal B(y)\|_1&=\sum_{n=-\infty}^\infty\|y_{n-1}-T^* y_n\|\\
&=\sum_{n=-\infty}^{-1}\|q^{n-1}x-q^{n}T^*x\|+\|q^{-1}x-T^*x\|
+\sum_{n=1}^{\infty}\|q^{-(n-1)}x-q^{-n}T^*x\|\\
&=\|q^{-1}x-T^*x\|\sum_{n=-\infty}^{-1}q^n+\|q^{-1}x-T^*x\|
+\|qx-T^*x\|\sum_{n=1}^{\infty}q^{-n}\\
&=\|q^{-1}x-T^*x\|\biggl(\,\sum_{n=1}^{\infty}q^{-n}+1\biggr)
+\|qx-T^*x\|\sum_{n=1}^{\infty}q^{-n}\\
&=\|q^{-1}x-T^*x\|\frac{q}{q-1}+\|qx-T^*x\|\frac{1}{q-1}.
\end{align*}
Substituting the $l_1$-sequence~$y=(y_n)_{n\in\mathbb Z}$ defined by~\eqref{eq:ytest} into~\eqref{eq:bbineq}, we obtain that for all $x\in X$ and $q>1$,
\begin{equation*}
\|q^{-1}x-T^*x\|q+\|qx-T^*x\|\geq\alpha(1+q)\|x\|.
\end{equation*}
From this, letting $q\to1$, we conclude that
\begin{equation*}
\|(I-T^*)x\|\geq\alpha\|x\|\qquad\text{for all $x\in X$}.
\end{equation*}
Thus, $I-T^*$ is bounded below which contradicts~\eqref{eq:notbb}.

\emph{Sufficiency.\,} Suppose that~\eqref{eq:specshad} is satisfied, i.e.,  $\lambda I-T$ is right invertible for every  $\lambda\in\mathbb C$ with $|\lambda|=1$. Since the right resolvent set of~$T$, the set of those $\lambda\in\mathbb C$ for which $\lambda I-T$ is right invertible, is an open subset of~$\mathbb C$ (see, e.g.,  \cite[Chap.~VII, Par.~2, Theorem~2.2, p.~192]{Con}), this implies that $\lambda I-T$ is right invertible in a neighborhood of the unit circle $|\lambda|=1$. Thus, there exists $\epsilon>0$ such that the holomorphic operator-valued function $H(\lambda)=\lambda I-T$ is right invertible for every $\lambda$ from the open and connected set $A_\epsilon\subset\mathbb C$ given by
\begin{equation*}
A_\epsilon=\{\,\lambda\in\mathbb C: 1-\epsilon<|\lambda|<1+\epsilon\,\}.
\end{equation*}
 By the application of Proposition~\ref{prop:holright} with $G=A_\epsilon$, we conclude that there exists a holomorphic function $R\colon A_\epsilon\rightarrow L(X)$ such that
 \begin{equation}\label{eq:rightinv}
 (\lambda I-T)R(\lambda)=I\qquad\text{for all $\lambda\in A_\epsilon$}.  
\end{equation}
The set $A_\epsilon$ is an annulus of the form~\eqref{eq:annul} with $\lambda_0=0$, $r_1=1-\epsilon$ and $r_2=1+\epsilon$. As noted in Sec.~\ref{sec:prelim},  the holomorphic function~$R\colon A_\epsilon\rightarrow L(X)$ can be expanded into a Laurent series
\begin{equation}\label{eq:rightlaur}
R(\lambda)=\sum_{n=-\infty}^\infty C_n\lambda^n,\qquad 1-\epsilon<|\lambda|<1+\epsilon,
\end{equation}
with unique coefficients $C_n\in L(X)$, $n\in\mathbb Z$. Substituting~\eqref{eq:rightlaur} into~\eqref{eq:rightinv}, we find that
\begin{equation*}
\sum_{n=-\infty}^\infty(C_{n-1}-TC_n)\lambda^n=I\qquad\text{whenever $1-\epsilon<|\lambda|<1+\epsilon$}.
\end{equation*}
From this, in view of the uniqueness of the coefficients of the Laurent series, we have that
\begin{equation*}
C_{-1}-TC_0=I
\end{equation*}
and 
\begin{equation*}
C_{n-1}-TC_n=0\qquad\text{for $n\neq0$}.
\end{equation*}
Hence,
\begin{gather}
	C_0=-T^{-1}(I-C_{-1}),\label{eq:crel1}\\
		C_n=T^{-n}C_0=-T^{-n-1}(I-C_{-1}),\qquad n=1,2,\dots\label{eq:crel2}
\end{gather}
and
\begin{equation}\label{eq:crel3}
C_{-n}=T^{n-1} C_{-1},\qquad n=1,2,\dots
\end{equation}
As noted in Sec.~\ref{sec:prelim} (see~\eqref{eq:limsup1} and~\eqref{eq:limsup2}), the convergence of the Laurent series~\eqref{eq:rightlaur} in the annulus~$A_\epsilon$ implies that
\begin{equation}\label{eq:mlimsup}
\limsup_{n\to\infty}\root{n}\of{\|C_n\|}\leq\frac{1}{1+\epsilon}<1
\end{equation}
and
\begin{equation}\label{eq:plimsup}
\limsup_{n\to\infty}\root{n}\of{\|C_{-n}\|}\leq 1-\epsilon<1.
\end{equation}
By virtue of~\eqref{eq:crel2} and~\eqref{eq:crel3}, we have that
\begin{equation*}
\root{n}\of{\|T^{-n}(I-C_{-1})\|}=
\root{n}\of{\|TC_n\|}
\leq\root{n}\of{\|T\|}\|\root{n}\of{\|C_n\|}
\end{equation*}
and
\begin{equation*}
\root{n}\of{\|T^{n}C_{-1}\|}=
\root{n}\of{\|TC_{-n}\|}\leq\root{n}\of{\|T\|}\root{n}\of{\|C_{-n}\|}
\end{equation*}
for $n=1,2,\dots$. Letting $n\to\infty$ in the last two inequalities and using~\eqref{eq:mlimsup} and~\eqref{eq:plimsup}, respectively, we obtain
\begin{equation*}
\limsup_{n\to\infty}\root{n}\of{\|T^{-n}(I-C_{-1})\|}\leq\frac{1}{1+\epsilon}<1
\end{equation*}
and
\begin{equation*}
\limsup_{n\to\infty}\root{n}\of{\|T^{n}C_{-1}\|}\leq 1-\epsilon<1.
\end{equation*}
Therefore, assumptions~\eqref{eq:contr1} and~\eqref{eq:contr2} of Proposition~\ref{prop:shadcrit2} are satisfied with the bounded linear operator~$B=C_{-1}$. By the application of Proposition~~\ref{prop:shadcrit2}, we conclude that~$T$ has the shadowing property.

\end{proof}
%\vskip10pt

\section*{Acknowledgements}
This work was supported in part by the Hungarian National Research, Development and Innovation Office grant no.~K139346.
The author would like to express his sincere thanks to the reviewers for their valuable remarks and suggestions, which have improved the presentation of the results.



\end{document}